\documentclass[12pt]{article}

\usepackage[english]{babel}
\usepackage{amsmath,amssymb,amsthm}
\usepackage{cite}
\usepackage{enumerate,mathtools}
\usepackage{todonotes}
\usepackage{nicefrac}
\usepackage{hyperref,xcolor,paralist}
\hypersetup{
colorlinks=true,
linkcolor=red,
citecolor = blue,      
}

\providecommand{\bysame}{\leavevmode\hbox to3em{\hrulefill}\thinspace}
\providecommand{\MR}{\relax\ifhmode\unskip\space\fi MR }

\providecommand{\href}[2]{#2}

\theoremstyle{plain}
\newtheorem{theorem}{{\bf{Theorem}}}[section]

\newtheorem*{assumption}{{\bf Assumption}}
\newtheorem{lemma}[theorem]{{\bf Lemma}}

\theoremstyle{definition}
\newtheorem*{remark}{{\bf {Remark}}}

\numberwithin{equation}{section}

\newcommand{\R}{\mathbb{R}}

\newcommand{\Z}{\mathbb{Z}}
\newcommand\blfootnote[1]{
	\begingroup
	\renewcommand\thefootnote{}\footnote{#1}
	\addtocounter{footnote}{-1}
	\endgroup
}
\newcommand{\im}{\mathrm{i}}

\newcommand{\E}{\mathbf{E}}

\hypersetup{
	colorlinks=true,
	linkcolor=red,
	citecolor = blue,      
}

\newcommand{\repsi}{\mathrm{Re}\Psi}
\newcommand{\impsi}{\mathrm{Im}\Psi}
\newcommand{\normxi}{\|\xi\|}
\newcommand{\Iinf}{\mathsf{I}_{\infty}}
\newcommand{\Izero}{\mathsf{I}_{0}}
\newcommand{\indl}{\mathsf{IND}_{\ell}}                                  
\newcommand{\indm}{\mathsf{IND}_{m}}                               
\newcommand{\indu}{\mathsf{IND}_{u}}   
\newcommand{\drift}{\mathrm{\mathbf{d}}}
\newcommand{\gauss}{\mathrm{\mathbf{Q}}} 
\newcommand{\jump}{\mathbf{\nu}}                           

\title{\bf{H\"{o}lder regularity for a class of nonlinear stochastic heat equations}}
\author{Sudheesh Surendranath  \thanks{Department of Mathematics, University of Utah \texttt{sudheesh@math.utah.edu}}}
\begin{document}
	\date{}
	\maketitle
    
	\begin{abstract}
	We investigate the sample path regularity of solutions to a class of stochastic partial differential equations of the form  
    \begin{equation*}
    \frac{\partial u}{\partial t}=\mathcal{L}u+b(u)+ \sigma(u)\dot{F} \, ,
    \end{equation*}
	 subject to a suitable initial condition. The noise term $\dot{F}$ is the generalized derivative of a random spatially homogeneous Gaussian distribution, which is white in time and colored in space, and $\mathcal{L}$ is the $L^{2}$-generator of a L\'evy process. Under certain growth assumptions on the characteristic exponent of the L\'{e}vy process, we derive sufficient conditions for the solution to be locally H\"older continuous jointly in space and time. Moreover, we show that these conditions are equivalent to those derived in related papers \cite{sanz2000path, khoshnevisan2022optimal, sanz2002holder}.
	\end{abstract}
	\section{Introduction}
	In this paper, we study the H\"{o}lder regularity of the solutions to a class of $(n+1)$-dimensional nonlinear stochastic heat equations. Consider
	\begin{equation}{\label{eqn:spde}}
	\frac{\partial u(t,x)}{\partial t}=(\mathcal{L}u)(t,x)+ b(u(t,x)) +\sigma(u(t,x))\dot{F}(t,x) \quad \text{for} \quad (t,x) \in (0,\infty) \times \mathbb{R}^{n} \, ,
	\end{equation}
	\begin{equation*}
	\text{with } u(0,x)=u_{0}(x), \quad x \in \mathbb{R}^{n}.
	\end{equation*}
	Here $\mathcal{L}$ denotes the infinitesimal generator of an $n-$dimensional L\'evy process with characteristic exponent $\Psi$. The function $u_{0}: \mathbb{R}^{n} \to \mathbb{R}$ is a bounded deterministic, $\rho-$H\"older continuous function and $b, \sigma: \mathbb{R} \to \mathbb{R}$ are Lipschitz continuous. The noise term $\dot{F}$ is the generalized derivative of random centered Gaussian Schwartz distribution. In particular, it is a centered $L^{2}(\Omega, \mathcal{F}, \mathbf{P})-$valued Gaussian process $\{F(\varphi): \varphi \in \mathcal{S}([0, \infty) \times \mathbb{R}^{n})\}$ with covariance structure
	\begin{equation}{\label{eqn:covarianceofnoise_1}}
	\mathrm{Cov}\left(F(\varphi_{1}),F(\varphi_{2})\right)=\int_{0}^{\infty}\int_{\mathbb{R}^{n}}(\varphi_{1} \ast \widetilde{\varphi_{2}})(t,x) \, \Gamma(dx) \, dt.
	\end{equation}
	Here, '$\ast$' denotes convolution in the spatial variable, $\widetilde{\varphi}(t,x) \coloneqq \varphi(t,-x)$ and $\Gamma$ is a tempered, non-negative and non-negative definite Borel measure on $\mathbb{R}^{n}$.
	Formally, one writes
	\begin{equation*}
	\mathrm{Cov}\left(\dot{F}(t_{1},x_{1}),\dot{F}(t_{2},x_{2})\right)=\delta_{0}(t_1-t_2)\Gamma(x_1-x_2) \, .
	\end{equation*}
	When $\Gamma=\delta_{0}$, $\dot{F}$ is usually termed \textit{white noise} in the literature. According to the Bochner-Minlos-Schwartz Theorem \cite{schwartz1966theorie}, $\Gamma$ is the Fourier transform of a non-negative, tempered measure. That is,
	\begin{equation}{\label{eqn:bochnerminlosschwarz}}
	\int_{\mathbb{R}^{n}} \varphi(x) \, \Gamma(dx)=\int_{\mathbb{R}^{n}}\widehat{\varphi}(\xi) \, \mu(d\xi) \, ,
	\end{equation}
	where 
	\begin{equation*}
	\widehat{\varphi}(\xi)=\int_{\mathbb{R}^{n}}\varphi(x)e^{i\langle \xi,x\rangle} \,dx \quad \text{and} \quad 	\int_{\mathbb{R}^{n}}\frac{\mu(d\xi)}{(1+\normxi^2)^k} < \infty \, ,
	\end{equation*}
	for some $k \in \mathbb{N}$. With this in mind, equation \eqref{eqn:covarianceofnoise_1} can be rewritten as 
	\begin{equation}{\label{eqn:covarianceofnoise_2}}
		\mathrm{Cov}\left(F(\varphi_{1}),F(\varphi_{2})\right)= \int_{0}^{\infty}\int_{\R^n}\widehat{\varphi}_{1}(t,\xi)\overline{\widehat{\varphi}_{2}(t,\xi)}\, \mu(d\xi) \, dt \, .
	\end{equation}

    Regarding the characteristic exponent $\Psi$, we impose the following condition throughout the paper:
    \begin{equation*}
    	\Psi^{-1}(0) = \{ 0\}
    \end{equation*}
    This is a non-degeneracy condition. It ensures that the underlying L\'{e}vy process is not concentrated on a proper subspace of $\R^n$. To see why, see \cite[Section~2.1]{khoshnevisan2022optimal}.
    
	When $\mathcal{L}= \Delta$, equation \eqref{eqn:spde} reduces to the classical stochastic heat equation, which has been extensively studied. The solution theory for such SPDEs is formulated using a theory of stochastic integration first developed by Walsh \cite{walsh1986introduction} and later extended by Dalang \cite{dalang1999extending}; see \cite{dalang2011stochastic} for a detailed survey of such integration theories and how they relate to the other notions of solutions. In these approaches, the solution map is a random-field $(t,x) \mapsto u(t,x)$ over space-time. Subsequent work by Foondun and Khoshnevisan \cite{foondun2013stochastic} extended this approach to more general generators $\mathcal{L}$. The optimal condition for the existence of a solution to \eqref{eqn:spde} is then given by:
	\begin{equation}{\label{eqn:dalangcondition}}
		\int_{\mathbb{R}^{n}}\frac{\mu(d\xi)}{1+\repsi({\xi})}<\infty \, .
	\end{equation}

    Once a solution to \eqref{eqn:spde} is established, a natural question to ask is about its sample path regularity. There is a large body of literature that discuss this question. See for example, the papers: Balan \cite{balan2011lp}, Balan-Jolis-Quer-Sardanyon \cite{balan2015holder}, Balan-Jolis-Quer-Sardanyon \cite{balan2016spdes}, Balan-Chen \cite{balan2018holder}, Balan-Quer-Sardanyon \cite{balan2019holder}, Balan-Song \cite{balan2016HyperbolicAM}, Chen-Dalang \cite{chen2014holder},Conus-Dalang \cite{conus2008holder}, Dalang-Sanz-Sol\'{e} \cite{dalang2009},  Dalang-Sanz-Sol\'{e} \cite{dalang2024}, Dalang-Zhang \cite{dalang2017holder}, Guo-Song-Song \cite{guo2024}, Hu-Lu-Nualart \cite{hu2013holder}, Hu-Huang-Nualart \cite{hu2014holder}, Hu-Huang-Nualart-Tindel \cite{hu2015stochastic}, Hu-Nualart-Song \cite{hu2013nonlinear}, Hu-Huang-L\^{e}-Nualart-Tindel \cite{hu2017stochastic}, Hu-L\^{e} \cite{hu2019joint}, Hu-Nualart-Xia \cite{hu2019holder} Khoshnevisan-Sanz-Sol\'{e} \cite{khoshnevisan2022optimal}, Millet-Sanz-Sol\'{e} \cite{millet1999},  Sanz-Sol\'{e}-Sarra \cite{sanz2000path}, Sanz-Sol\'{e}-Sarra \cite{sanz2002holder},  Sanz-Sol\'{e}-Vuillermot \cite{sanzsole2003}, Sanz-Sol\'{e}-Vuillermot \cite{sanzsole2009}, Song \cite{song2017} and Song-Song-Xu \cite{song2020fractional}. 
    
    The approach we will use is to use a version of the Kolmogorov continuity theorem (\cite{khoshnevisan2014analysis} Appendix C.2, \cite{khoshnevisan2009primer}). This requires establishing moment estimates of the increments of the solution. Before we describe some results which purse the same approach, we introduce some definitions. 
    \begin{equation*}
    	\begin{split}
    		\indu & \coloneqq \sup \left\{ \eta \in (0,1) : \int_{\R^n}\frac{\mu(d\xi)}{(1 + \repsi(\xi))^{1-\eta}} < \infty \right\}, \\
    		\indm & \coloneqq \sup \left\{ \gamma \in (0,1) : \int_{\R^n} \frac{|\Psi(\xi)|^{\gamma}\mu(d\xi)}{1 + \repsi(\xi)} < \infty \right\}, \qquad \text{and} \\
    		\indl & \coloneqq \sup \left\{ \delta \in (0,1) : \int_{\R^n} \frac{\| \xi \|^{2\delta}\mu(d\xi)}{1 + \repsi(\xi)} < \infty \right\}
    	\end{split}
    \end{equation*}
    with $\sup \emptyset \coloneqq 0$. The subscripts stand for \textit{upper}, \textit{middle} and \textit{lower} respectively. These \textit{fractal indices} were introduced in Khoshnevisan-Sanz-Sol\'{e} \cite{khoshnevisan2022optimal} in relation to the H\"{o}lder exponents of the solution to a particular case of the SPDE \eqref{eqn:spde}. Using a certain growth estimate on $\Psi$ (see equation \eqref{eqn:khinchtineestimate}), it is not hard to verify that
    \begin{equation*}
    	\indl \leqslant \indm \leqslant \indu .
    \end{equation*}
    However, the inequality can be strict (see for example \cite[Proposition~2.1]{khoshnevisan2022optimal}). 
    
	For the classical case $\mathcal{L}=\Delta$ (that is, $\Psi(\xi)=\repsi(\xi)=\normxi^2$), Sanz-Sol\'e and Sarra \cite{sanz2000path, sanz2002holder} have proven that the following condition:
	\begin{equation}
	\int_{\mathbb{R}^{n}}\frac{\mu(d\xi)}{(1+\normxi^{2})^{1-\eta}}< \infty \quad \text{for some } \eta \in (0,1)
	\end{equation}
	is sufficient for the H\"older continuity of the solution to \eqref{eqn:spde}. In terms of the fractal indices, this is equivalent to the condition that $\indu > 0$. In particular, Sanz-Sol\'{e} and Sarra have shown that $u \in \mathcal{C}^{\alpha, \beta}_{\operatorname{loc}}(\R_{> 0} \times \R^n)$ for all $\alpha < \indu$ and $\beta < \nicefrac{\indu}{2}$. In \cite{khoshnevisan2022optimal}, Khoshnevisan and Sanz-Sol\'e have established sufficient and necessary conditions for the H\"older regularity to \eqref{eqn:spde} in the case when $\sigma \equiv 1$, or the \textit{additive noise} case. In particular, they showed that if 
	\begin{equation}{\label{eqn:optimal1}}
	\int_{\mathbb{R}^{n}}\frac{|\Psi(\xi)|^{\gamma}\mu(d\xi)}{1+\repsi(\xi)} < \infty \quad \text{for some } \gamma \in (0,1) \, ,
	\end{equation} then the solution to \eqref{eqn:spde} is locally H\"older continuous in time. Under the stronger condition,
	\begin{equation}{\label{eqn:optimal2}}
	\int_{\mathbb{R}^{n}}\frac{\normxi^{2\delta} \mu(d\xi)}{1+\repsi(\xi)} < \infty \quad \text{for some } \delta \in (0,1) \,
	\end{equation} the solution is also H\"older continuous in space. Moreover, these conditions are optimal in the sense that they are sufficient as well as necessary. In terms of the fractal indices, condition \eqref{eqn:optimal1} is equivalent to that $\indm > 0$ and condition \eqref{eqn:optimal2} to $\indl > 0$. In this case, the solution $u \in \mathcal{C}^{\alpha, \beta}_{\operatorname{loc}}(\R_{>0} \times \R^n)$ for all $\alpha < \nicefrac{\indm}{2}$ and $\beta < \indl$.

   The goal of this paper is to derive sufficient conditions for the H\"older regularity of the random-field solution to \eqref{eqn:spde}. In the present setting, the function $\sigma$ can be nonlinear.
    Although the result of \cite{khoshnevisan2022optimal} holds for any characteristic exponent (and hence any $\mathcal{L}$), our results are only valid when $\Psi$ satisfies some growth condition near zero and infinity (see Assumption in the next section). 
    

     The following is the main theorem of this paper:
	\begin{theorem}{\label{thm:main_thm}}
		Let $\{u(t,x): (t,x) \in [0, \infty) \times \mathbb{R}^{n}\}$ be the random-field solution to the SPDE \eqref{eqn:spde}. Under Assumptions \ref{assumption1} and \ref{assumption2}, the following equality holds
		\begin{equation*}
			\indl = \indm = \indu .
		\end{equation*}
	    In this case, if
		\begin{equation}{\label{eqn:sanz-sole_sarra}}
		\int_{\mathbb{R}^{n}}\frac{\mu(d\xi)}{(1+\repsi(\xi))^{1-\eta}} < \infty \quad \text{for some } \eta \in (0,1) \, ,
		\end{equation}
		i.e.\ if $\indu > 0$ (and hence $\indm, \indl > 0$), then the function $(t,x) \to u(t,x)$ is a.s locally H\"older continuous on $\R_{>0} \times \R^n$.
		In other words, the conditions \eqref{eqn:sanz-sole_sarra},\eqref{eqn:optimal1} and \eqref{eqn:optimal2} are all equivalent and hence any one of them is sufficient for the H\"older continuity of the solution.
	\end{theorem}  
    The key step in the proof involves establishing certain $L^1$ bounds (see Lemma \ref{lem:L1_bounds}) on the transition densities of the L\'evy process. These estimates rely critically on the growth assumptions on the characteristic exponent $\Psi$. The H\"{o}lder indices that we obtain do not seem to optimal as suggested by those in \cite{khoshnevisan2022optimal}. Our methods also do not seem to translate to the wave equation. We leave these questions for future work. 
    
    The paper is organized as follows: in Section \ref{sec:levyprocesses}, we state our main assumptions and derive the aforementioned $L^1$ bounds on the transition densities. In Section \ref{sec:holdercontinuity}, we prove Theorem \ref{thm:main_thm} using the results proven in Section \ref{sec:levyprocesses}.
	\begin{remark}
	 For two functions $f$ and $g$, we write $f \lesssim g$ if there exists a constant $C>0$, independent of the variables of $f$ and $g$, such that $f \leqslant Cg$. In other words, $f$ is bounded above by $g$ upto a universal constant factor. 
	\end{remark}

\section{Preliminaries}{\label{sec:levyprocesses}}
    In this section, we will introduce our main Assumptions on our L\'{e}vy process and use them to establish the necessary lemmas for proving Theorem \ref{thm:main_thm}. These lemmas rely on certain moment estimates of the underlying L\'{e}vy process. For this reason, let us first recall some basic properties. Consider a process $\{X_{t}\}_{t \geqslant0}$ on $\mathbb{R}^{n}$. By definition, it is a L\'{e}vy process if $X_{0} = 0$ a.s., the paths $t \mapsto X_{t}$ are c\`adl\`ag, and the process has stationary and independent increments. A fundamental result, the L\'{e}vy–Khintchine formula \cite{ken1999levy}, tells us that the characteristic function of $X_{t}$ is uniquely determined by a continuous function $\Psi: \mathbb{R}^{n} \to \mathbb{C}$, known as the \emph{characteristic exponent} of $X$. That is,
	\begin{equation*}
	\mathbf{E}\left(e^{\im \langle \xi, X_{t}\rangle}\right)=e^{-t\Psi(\xi)} \, ,
	\end{equation*}
	where
	\begin{equation*}
	\Psi(\xi)=\im (\langle \drift, \xi \rangle)+\frac{1}{2}\langle \xi,\gauss\xi \rangle+\int_{\mathbb{R}^{n}}\left(1-e^{\im \langle \xi,x \rangle}+ \im \langle \xi,x\rangle \mathbf{1}_{\{\|x\| \leqslant 1\}}\right)\jump(dx) \, .
	\end{equation*}
	Here, the vector $\drift \in \mathbb{R}^{n}$ is called the \emph{drift vector}, and it captures the deterministic linear trend of the process. The matrix $\gauss$ is positive-definite and acts as the covariance matrix of the Gaussian component, describing the continuous fluctuations. Finally, $\jump$ is a $\sigma$-finite Borel measure, known as the \emph{Lévy measure}, which accounts for the jumps of the process. It satisfies $\jump(\{0\}) = 0$ and
	\begin{equation*}
		\int_{\R^n}(1 \wedge \|x\|^2) \, \jump(dx) < \infty
	\end{equation*}
    ensuring that small jumps do not accumulate too wildly. Together, the triplet $(\drift, \gauss, \jump)$ uniquely determines the characteristic exponent $\Psi$. An important and useful consequence of this representation is the following growth estimate for $\Psi$, which will play a recurring role in our arguments:
	\begin{equation}{\label{eqn:khinchtineestimate}}
	\sup_{\xi \in \mathbb{R}^{n}}\frac{|\Psi(\xi)|}{1+\|\xi\|^{2}}< \infty.
	\end{equation}
    
     We now state our assumptions on the characteristic exponent of the underlying L\'evy process:
	\begin{assumption}{\label{assumption}} Assume $\Psi$ satisfies the following conditions:
		\begin{enumerate}
			\item \label{assumption1} $\liminf_{\normxi \to \infty}\frac{\log \repsi(\xi)}{\log \normxi}= \Iinf>0 \text{ for some } \Iinf \in (0,2)$ ; \\
			\item \label{assumption2} $\int_{\{\normxi \leqslant 1\}}\frac{|\Psi(\xi)| }{\normxi^{n+\Izero}}\,d\xi < \infty \text{ for some } \Izero \in (0,1)$ ; \\ 
		\end{enumerate}
	\end{assumption}
	
    The constant $\Iinf$, known as the Blumenthal–Getoor index \cite{blumenthal1961sample}, characterizes the growth of $\Psi(\xi)$ at infinity, in the sense that $\repsi(\xi) \gg \|\xi\|^{\Iinf}$ when $\|\xi\|$ is large. The index $\Izero$ serves as an analogous measure near the origin, in that $|\Psi(\xi)|$ grows no faster than $\|\xi\|^{\Izero - \varepsilon}$ near the origin, for any arbitrarily small $\varepsilon > 0$. This condition ensures the existence of moments of $X_{t}$ (see Lemma \ref{lem:momentequivalence}). 
    
	As mentioned in the Introduction, Assumptions \ref{assumption1} and \ref{assumption2} play a key role in obtaining certain $L^1$ estimates for the transition densities. The proof of these estimates depends on the regularity properties of the transition densities and on moment bounds for the underlying Lévy process, which we discuss next.
	
	\begin{lemma}{\label{lem:smoothtransitiondensitities}} Under Assumption \ref{assumption1}, the L\'evy process $\{X_{t}\}_{t \geqslant0}$ admits a family of transition functions $p_{t}: \R^n \to \R^n$ for $t \geqslant0$. Moreover, for every $r >0$ and $\theta \in [0,1]$, $p_{\cdot} \in \bigcap_{k=0}^{\infty} \bigcap_{\alpha \in \Z_{\geqslant0}^{n+1} , |\alpha| = k} \mathcal{C}^{k,\theta}([r, \infty) \times \R^n)$
		
		\begin{proof}
			Fix $r>0$ and $\theta \in [0,1]$. For a multi-index $\alpha = (\alpha_0, \alpha_1, \ldots \alpha_n) \in \Z_{\geqslant0}^{n+1}$ with $|\alpha| = k$, let $\partial^{\alpha} \coloneqq \partial_{t}^{\alpha_0}\partial_{x_1}^{\alpha_1} \cdots \partial_{x_n}^{\alpha_n}$. By Assumption \ref{assumption1}, we know that outside a neighborhood of the origin, say, on the set $\{\|\xi\| > M\}$, the inequality $|e^{-t\Psi(\xi)}| = |e^{-t\repsi(\xi)}| \leqslant e^{-t \normxi^{\Iinf}}$  holds. Together with the continuity of $\Psi$, this implies that
			\begin{equation*}
			\int_{\mathbb{R}^{n}}\left|e^{-t\Psi(\xi)}e^{-\im \langle x, \xi \rangle}\right|d\xi = \int_{\mathbb{R}^{n}}\left|e^{-t\Psi(\xi)}\right|d\xi \leqslant \int_{\{\normxi \leqslant M\}}e^{-t\repsi(\xi)}d\xi+\int_{\{\normxi>M\}}e^{-t\normxi^{\Iinf}}d\xi < \infty \, .
			\end{equation*}
			Hence, by the Fourier inversion theorem, the density of $X_{t}$ can be expressed as
			\begin{equation*}
			p_{t}(x)=\frac{1}{(2\pi)^{n}}\int_{\mathbb{R}^{n}}e^{-t\Psi(\xi)}e^{-\im \langle x, \xi \rangle} \, d\xi
			\end{equation*}
			Moreover, using the growth estimate in equation \eqref{eqn:khinchtineestimate} and the reasoning above, we also obtain
			\begin{align*}
			\int_{\mathbb{R}^{n}}\left|\partial^{\alpha}\left(e^{-t\Psi(\xi)}e^{\im \langle x, \xi \rangle}\right)\right|d\xi & \lesssim \int_{\mathbb{R}^{n}}e^{-t\repsi(\xi)} \normxi^{3k}d\xi \\
			&=\int_{\{ \normxi \leqslant  M\}}e^{-t\repsi(\xi)} \normxi^{3k}d\xi + \int_{\{\normxi> M\}}e^{-t\normxi^{\Iinf}} \normxi^{3k}d\xi < \infty \\
			& \lesssim 1 + \frac{1}{t^{\frac{3k+1}{\Iinf}}}\int_{\R^n}e^{-\normxi^{\Iinf}}\normxi^{3k} \, d\xi
			\end{align*} 
		    The right-hand side is uniformly bounded for all $t \geqslant r$ and $x \in \mathbb{R}^{n}$. Therefore, by the Dominated Convergence Theorem, we obtain 
		    \begin{equation*}
		    	\sup_{t \geqslant r}\sup_{x \in \R^n} | \partial^{\alpha}p_{t}(x)| < \infty.
		    \end{equation*}
	    	We now turn to the H\"older continuity of the transition functions. By the triangle inequality, 
	    	\begin{equation*}
	    		 |\partial^{\alpha}p_{t+\varepsilon}(x + h) - \partial^{\alpha}p_{t}(x)| \leqslant | \partial^{\alpha}p_{t}(x + h) - \partial^{\alpha}p_{t}(x)| + | \partial^{\alpha}p_{t+\varepsilon}(x) - \partial^{\alpha}p_{t}(x)|.
	    	\end{equation*}
    	    We estimate each term individually. For the first term, 
			\begin{equation*}
					| \partial^{\alpha}p_{t}(x+h)-\partial^{\alpha}p_{t}(x)|   \lesssim \int_{\R^n}e^{-t\repsi(\xi)}\|\xi \|^{k}|e^{\im \langle h, \xi \rangle} - 1 | \, d\xi \lesssim \|h\|^{\theta} \int_{\R^n}e^{-t\repsi(\xi)}\|\xi \|^{k + \theta} \, d\xi
			\end{equation*}
		    As before, we split the integrand over $\{\|\xi\| \leqslant M\}$ and $\{\|\xi\| > M\}$ and use Assumption \ref{assumption1} to obtain
			\begin{align}
				| \partial^{\alpha}p_{t}(x+h)-\partial^{\alpha}p_{t}(x)|    
				&  \lesssim \|h\|^{\theta}\int_{\{\normxi \leqslant M\}}e^{-t\repsi(\xi)}\normxi^{k +\theta}  \,  d\xi  + \|h\|^{\theta}\int_{\{\normxi > M\}}e^{-t\normxi^{\Iinf}}\normxi^{k + \theta} \, d\xi  \nonumber \\
				& \lesssim 1 + \frac{\|h\|^{\theta}}{t^\frac{\theta+k+1}{\Iinf}} \int_{\mathbb{R}^{n}}e^{-\normxi^{\Iinf}} \normxi^{k + \theta} \, d\xi \nonumber \\
				& \lesssim \frac{\|h\|^\theta}{t^{\frac{\theta+k+1}{\Iinf}}}. \label{eqn:spatial_sup_bound}
		   \end{align}
	       Similarly, we use  the inequality $|e^{- \varepsilon \repsi(\xi)} - 1| \leqslant \varepsilon^{\theta}|\Psi(\xi)|^{\theta}$ along with equation \eqref{eqn:khinchtineestimate} to derive
	       \begin{align}
	       	| \partial^{\alpha}p_{t+\varepsilon}(x)-\partial^{\alpha}p_{t}(x)|    
	       	&  \lesssim |\varepsilon|^{\theta}\int_{\{\normxi \leqslant M\}}e^{-t\repsi(\xi)}\normxi^{k}  \,  d\xi  + |\varepsilon|^{\theta}\int_{\{\normxi > M\}}e^{-t\normxi^{\Iinf}}\normxi^{k + 2\theta} \, d\xi \nonumber \\
	       	& \lesssim 1 + \frac{|\varepsilon|^{\theta}}{t^\frac{2\theta+k+1}{\Iinf}} \int_{\mathbb{R}^{n}}e^{-\normxi^{\Iinf}} \normxi^{k + 2\theta} \, d\xi \nonumber \\
	       	& \lesssim \frac{|\varepsilon|^\theta}{t^{\frac{2\theta+k+1}{\mathrm{\Iinf}}}}. {\label{eqn:temporal_sup_bound}}
	       \end{align}
           Putting the two inequalities together, we see that
           \begin{equation*}
           	\sup_{\substack{t \geqslant r \\ \varepsilon \neq 0}} \sup_{\substack{x \in \R^n \\ h \neq 0}}  \frac{|\partial^{\alpha}p_{t+\varepsilon}(x + h) - \partial^{\alpha}p_{t}(x)|}{|\varepsilon|^{\theta} + \|h\|^{\theta}} < \infty
           \end{equation*}
           This completes the proof of the Lemma.
	     \end{proof}	
	\end{lemma}

    We now turn to moment estimates for the process $\{X_{t}\}_{t \geqslant 0}$ . Sato\cite{ken1999levy} presents sufficient and necessary condition for the existence of moments for a L\'evy process in terms of the L\'evy measure. For our purpose, it is convenient to derive equivalent conditions in terms of the characteristic exponent. 
 
 \begin{lemma}{\label{lem:momentequivalence}}
 	Under Assumption \ref{assumption2}, the following are equivalent:
 	\begin{enumerate}
 		\item $\mathbf{E}(\|X_{t}\|^{\Izero}) < \infty$ ; \\ \label{moment1}
 		\item $\int_{\{\|x\|>1\}}\|x\|^{\Izero}\,\jump(dx) < \infty$ ;\\ \label{moment2}
 		\item $\int_{\{\normxi \leqslant 1\}}\frac{\repsi(\xi)}{\normxi^{n+\Izero}}\,d\xi < \infty$ ;\\ \label{moment3}
 		\item $\int_{\{\normxi \leqslant 1\}}\frac{|\Psi(\xi)|}{\normxi^{n+\Izero}} \,d\xi < \infty$ . \\  \label{moment4}
 	\end{enumerate}
 \end{lemma}

 \begin{proof}
 	The implication $\eqref{moment1} \iff \eqref{moment2}$ follows from \cite[p.~159]{ken1999levy}. For $\eqref{moment2} \implies \eqref{moment3}$, observe that 
 	\begin{equation*}
 	\operatorname{Re} \Psi(\xi)=\frac{1}{2}\langle\xi, \gauss \xi \rangle
 	+\int_{\R^n}\left(1-\cos (\langle\xi, x \rangle) \right) \,\jump(dx) \, .
 	\end{equation*}
 	Since $\gauss$ is symmetric and positive-definite, there exists a largest non-negative eigenvalue $\operatorname{\lambda}_{\operatorname{max}}$ such that $\langle \xi, \gauss\xi \rangle \leqslant \operatorname{\lambda}_{\operatorname{max}}\normxi^{2} \,\, \forall \,\, \xi \in \R^n$.
 	Also, by Taylor's formula, we have
 	\begin{equation*}
 	1-\cos (\langle \xi, x \rangle) < 2\left(1 \wedge\|\xi\|^2\|x\|^2 \right) \quad \forall \xi, x \in \R^n \,
 	\end{equation*}
 	Hence,
 	\begin{equation*}
 	\int_{\{\|\xi\| \leqslant 1\}} \frac{\repsi(\xi)}{\|\xi\|^{n+\Izero}} \, d\xi	
 	\leqslant \frac{\operatorname{\lambda}_{\operatorname{max}}}{2} \int_{\{\|\xi\| \leqslant 1\}}\|\xi\|^{2-n-\Izero} \, d\xi +\int_{\{\normxi \leqslant 1 \}}\int_{\R^n}\frac{\left(1 \wedge \normxi^{2} \|x\|^{2} \right)}{\normxi^{n+\Izero}} \, \jump(dx) \, d\xi 
 	\end{equation*}
    Split the second term into $\{\|x\|\le 1\}$ and $\{\|x\|>1\}$, and on $\{\|x\|>1\}$ split further by whether $\|\xi\|\|x\|\leqslant 1$ or $>1$. Doing so transforms the right side of the above equation to 
    \begin{align*}
    	\frac{\operatorname{\lambda}_{\operatorname{max}}}{2} 
    	&\int_{\{\|\xi\| \leqslant 1\}}\|\xi\|^{2-n-\Izero} \, d\xi + \int_{\{\|\xi\| \leqslant 1\}}\int_{\{\|x\| \leqslant 1\}}\|\xi\|^{2-n-\Izero}\|x\|^2 \,\jump(dx)\,d\xi \\
    	& \, + \int_{\{\frac{1}{\|x\|} \leqslant \normxi \leqslant 1\}}  \int_{\left\{\|x\|>1 \right\}}\frac{1}{\|\xi\|^{n+\Izero}} \, \jump(dx) \, d\xi  + \int_{\left\{\|\xi\| \leqslant \frac{1}{\|x\|} \right\}} \int_{\{\|x\|>1\}}\|\xi\|^{2-n-\Izero}\|x\|^2 \, \jump(dx) \, d\xi.
    \end{align*}
 	The first three integrals in the above line are finite on account of local integrability in the $\xi$ variable (recall $\Izero-2<0$) and the definition of the L\'{e}vy measure $\nu$. As for the fourth integral, a change to polar coordinates yields
 	\begin{equation*}
 	\int_{\{\|x\|>1\}}\|x\|^{2} \, \jump(dx) \int_{0}^{\frac{1}{\|x\|}} r^{1-\Izero} \, dr= \frac{1}{2-\Izero}\int_{\{\|x\|>1\}}\|x\|^{\Izero} \, \jump(dx) \,  < \infty .
 	\end{equation*}
 	This proves $\eqref{moment2} \implies \eqref{moment3}$. For the implication $\eqref{moment3} \implies \eqref{moment2}$, due to the positive-definiteness of $\gauss$, we have
 	\begin{align*}
 	\int_{\{\|\xi\| \leqslant 1\}}\frac{\repsi(\xi)}{\normxi^{n+\Izero}} \, d\xi & \geqslant\int_{\{\normxi \leqslant 1\}}\int_{\R^n}\frac{1-\cos(\langle \xi, x \rangle)}{\normxi^{n+\Izero}} \, \jump(dx) \, d\xi \\
 	&=\int_{\R^n}\|x\|^{\Izero}\jump(dx)\int_{\{\|\xi\|\leqslant \|x\|\}}\frac{1-\cos \left( \langle \xi,\frac{x}{\|x\|}\rangle \right)}{\normxi^{n+\Izero}} \, d\xi \, ,
 	\end{align*}
 	where, in the equality, we have performed a change of variables from $\xi \to \frac{\xi}{\|x\|}$. For $\|x\| > 1$, the inner integral above can be bounded below by a constant 
 	\begin{equation*}
 		\mathcal{C} \coloneqq \int_{\{\normxi \leqslant 1\}}\frac{1-\cos(\langle \xi,\frac{x}{\|x\|} \rangle)}{\|\xi\|^{n+\Izero}} \, d\xi \lesssim \int_{\{\normxi \leqslant 1\}}\frac{1}{\|\xi\|^{n+\Izero-2}} \, d\xi < \infty \, .
 	\end{equation*}
 	Thus, 
 	\begin{align*}
 		\mathcal{C}\int_{\{\|x\|>1\}}\|x\|^{\Izero} \, \jump(dx) 
 		& = \int_{\{\|x\|>1\}}\|x\|^{\Izero} \, \jump(dx)\int_{\{\|\xi\|\leqslant 1\}}\frac{1-\cos(\langle \xi,\frac{x}{\|x\|} \rangle)}{\|\xi\|^{n+\Izero}} \, d\xi \\
 		& \leqslant 	\int_{\R^n}\|x\|^{\Izero} \, \jump(dx)\int_{\{\|\xi\|\leqslant \|x\|\}}\frac{1-\cos \big( \langle \xi,\frac{x}{\|x\|}\rangle\big)}{\normxi^{n+\Izero}} \, d\xi <\infty,
 	\end{align*}
    by assumption. This proves $\eqref{moment3} \implies \eqref{moment2}$.
    
   Since \eqref{moment4} $\implies$ \eqref{moment3} follows immediately, the steps we just established also show that \eqref{moment4} $\implies$ \eqref{moment2}. To complete the chain of equivalences, it remains to prove \eqref{moment2} $\implies$ \eqref{moment4}. Because we already showed that \eqref{moment2} $\implies$ \eqref{moment3}, it is enough to check that \eqref{moment2} also implies that
   \begin{equation*}
   	 \int_{\{\normxi \leqslant 1\}}\frac{\impsi(\xi)}{\normxi^{n+\Izero}} \, d\xi < \infty
   \end{equation*}
   To this end, note that
 	\begin{equation*}
 		\impsi(\xi) = \langle \drift, \xi \rangle + \int_{\{ \|x\| \leqslant 1\}}(\langle \xi, x\rangle) - \sin(\langle x, \xi \rangle) \, \nu(dx) +  \int_{\{ \|x\| > 1\}}\sin(\langle \xi, x\rangle) \, \nu(dx)
 	\end{equation*}
    On the region $\{\normxi \leqslant 1\} \times \{ \|x\| \leqslant 1\}$, Taylor's formula gives
    \begin{equation*}
    	|\langle \xi, x\rangle) - \sin(\langle x, \xi \rangle| \lesssim \normxi^3 \|x\|^3
    \end{equation*}
    On the other region $\{\normxi \leqslant 1\} \times \{ \|x\| > 1\}$, we use the elementary bound $\sin(\langle \xi, x \rangle) \lesssim (1 \wedge \normxi \|x\|)$. Thus. 
    \begin{align*}
    	\int_{\{\normxi \leqslant 1\}} \frac{\impsi(\xi)}{\normxi^{n + \Izero}} \, d\xi \leqslant \|\drift\| \int_{\{\normxi \leqslant 1\}}\normxi^{1- n -\Izero} \, d\xi 
    	& + \int_{\{\normxi \leqslant 1\}}\int_{ \{ \|x\| \leqslant 1\}} \normxi^{3-n-\Izero} \|x\|^3 \, \nu(dx) \, d\xi \\
    	& + \int_{\{\normxi \leqslant 1\}}\int_{ \{ \|x\| > 1\}}\frac{(1 \wedge \normxi \|x\|)}{\normxi^{n+\Izero}} \, \nu(dx) \, d\xi
    \end{align*}
    The first two terms are finite by local integrability in the $\xi$ variable (since $\Izero<1$) and the definition of the L\'evy measure. For the last term, up to a constant depending only on $n$, we write in polar coordinates:
    \begin{align*}
    	\int_{ \{ \|x\| > 1 \}} \int_{0}^{1} (1 \wedge r\|x\|)r^{-1-\Izero} \, dr \, \nu(dx) 
    	& = \int_{ \{ \|x\| > 1 \}}  \int_{0}^{\frac{1}{\|x\|}} r\|x\| r^{-1-\Izero} \, dr \, \nu(dx)  \\
     	& \qquad \qquad \qquad + \int_{ \{ \|x\| > 1 \}}  \int_{\frac{1}{\|x\|}}^{1} r^{-1-\Izero} \, dr \, \nu(dx) \\
     	& = \frac{1}{\Izero (1-\Izero)}\int_{\{ \|x\| > 1\}} \|x\|^{\Izero} \, \nu(dx) -\frac{1}{\Izero}\nu(\{ \|x\| > 1\})
    \end{align*}
    Both terms on the right are finite by \eqref{moment2} and the integrability of $\nu$.
    Hence the integral of $\impsi$ is finite, which proves \eqref{moment2} $\implies$ \eqref{moment4}.
    This completes the proof.
	\end{proof}	
    Now we use the above Lemma to obtain a growth estimate for the moments of $X_{t}$ as a function of $t$. Using the stationary and independent increments property of $X_{\cdot}$, it is not hard to show that $\mathbf{E}(\|X_{t}\|^{2}) \lesssim t$, provided the expectation exists. The following lemma provides a precise quantification of this result in our context, where $X_{\cdot}$ need not be square integrable.
	\begin{lemma}{\label{lem:momentgrowth1}}
		There exist $C(\Izero)>1$ such that,
		\begin{equation*}
		\mathbf{E}(\|X_{t}\|^{\Izero}) \leqslant C(\Izero)\int_{\mathbb{R}^{n}}\frac{\left( t|\Psi(\xi)|\wedge 1 \right)}{\normxi^{n+\Izero}} \, d\xi \, .
		\end{equation*}
	\end{lemma}
	\begin{proof}
		Recall as in proof of Lemma \ref{lem:momentequivalence},
		\begin{equation*}
		\mathcal{C}= \int_{\mathbb{R}^{n}}\frac{1-\cos(\langle \xi,\frac{x}{\|x\|}\rangle)}{\|\xi\|^{n+\Izero}}\, d\xi < \infty \, .
		\end{equation*}
		By a change of variables $\xi \to \|x\|\xi$,
		\begin{equation*}
		\|x\|^{\Izero}=\frac{1}{\mathcal{C}} \int_{\mathbb{R}^{n}}\frac{1-\cos(\langle \xi,x\rangle)}{\|\xi\|^{n+\Izero}} \, d\xi \quad \forall x \in \mathbb{R}^{n} \, .
		\end{equation*}
		Setting $x=X_{t}$, we take expectation and use the elementary inequality $|1-e^{-t\Psi(\xi)}|\leqslant t|\Psi(\xi)|\wedge 1$ 
		\begin{equation*}
		\mathbf{E}(\|X_{t}\|^{\Izero}) \lesssim \int_{\mathbb{R}^{n}}\frac{\left(t|\Psi(\xi)|\wedge 1 \right)}{\normxi^{n+\Izero}} \, d\xi \, .
		\end{equation*}
	\end{proof}
    \noindent For the proof of the H\"{o}lder estimates, we will need a sharper growth estimate on the moments. This is the content of the next lemma. 
	\begin{lemma} {\label{lem:moment_growth2}}
		For every $\kappa \in (0, \nicefrac{\Izero}{2})$, there exists $C(\kappa)>0$ such that
		\begin{equation*}
		\mathbf{E}\left(\|X_{t}\|^{\Izero}\right) \leqslant C(\kappa)t^{\kappa} \, .
		\end{equation*}
	\end{lemma}
	\begin{proof}
		Define $\mathcal{I}(t) \coloneqq \sup_{0 \leqslant s \leqslant t}\mathbf{E}\left(\|X_{s}\|^{\Izero}\right)$. From the previous Lemma, we know that
		\begin{align*}
		\int_{0}^{1}\frac{\mathcal{I}(t)}{t^{1+ \kappa}} \, dt & \lesssim \int_{0}^{1}\frac{dt}{t^{1+\kappa}}\int_{\mathbb{R}^{n}}\frac{(t|\Psi(\xi)| \wedge 1)}{\normxi^{n+\Izero}} \, d\xi \\
		&= \int_{0}^{1}\int_{\{\xi: |\Psi(\xi)| \leqslant 1\}}\frac{(t|\Psi(\xi)| \wedge 1)}{t^{1+\kappa}\normxi^{n+\Izero}} \,d\xi \, dt+\int_{0}^{1}\int_{\{\xi: |\Psi(\xi)| > 1\}}\frac{(t|\Psi(\xi)| \wedge 1)}{t^{1+\kappa}\normxi^{n+\Izero}} \, d\xi \, dt \, .
		\end{align*}
	    On the set $\{\xi: |\Psi(\xi)| \leqslant 1\}$, which is bounded and contains the origin by Assumption~\ref{assumption1}, we also have $t|\Psi(\xi)| < t < 1$. Therefore, by Assumption \ref{assumption2}, the first integral is finite;
		\begin{equation*}
		\int_{0}^{1}\int_{\{\xi: |\Psi(\xi)| \leqslant 1\}}\frac{(t|\Psi(\xi)| \wedge 1)}{t^{1+\kappa}\normxi^{n+\Izero}} \, d\xi \, dt=\frac{1}{1-\kappa} \int_{\{\xi: |\Psi(\xi)| \leqslant 1\}}\frac{|\Psi(\xi)|}{\normxi^{n+\Izero}} \, d\xi < \infty \, .
		\end{equation*}
		As for the other integral, we split the $dt$ integration according to whether $t|\Psi(\xi)| \leqslant 1$ or $>1$. Also, note that by Assumption \ref{assumption1},  the inclusion $\{\xi: |\Psi(\xi)|>1\} \subset \{\xi: \normxi \geqslant M \}$ holds for some large $M \gg 1$. Therefore, 
		\begin{align*}
		\int_{0}^{1}\int_{\{\xi: |\Psi(\xi)| > 1\}}\frac{(t|\Psi(\xi)| \wedge 1)}{t^{1+\kappa}\normxi^{n+\Izero}} \,  d\xi \, dt 
		& =\int_{\{\xi: |\Psi(\xi)| > 1\}}\int_{0}^{\frac{1}{|\Psi(\xi)|}}\frac{(t|\Psi(\xi)| \wedge 1)}{t^{1+\kappa}\normxi^{n+\Izero}} \, d\xi \,  dt \\
	    & \qquad \qquad \qquad +\int_{\{\xi: |\Psi(\xi)| > 1\}}\int_{\frac{1}{|\Psi(\xi)|}}^{1}\frac{(t|\Psi(\xi)| \wedge 1)}{t^{1+\kappa}\normxi^{n+\Izero}} \, d\xi \, dt  \\
		& \leqslant \frac{1}{1- \kappa}\int_{\{\normxi \geqslant M\}}\frac{|\Psi(\xi)|^{\kappa}}{\normxi^{n+\Izero}} \, d\xi+\frac{1}{\kappa}\int_{\{\normxi \geqslant M\}}\frac{|\Psi(\xi)|^{\kappa}}{\normxi^{n+\Izero}} \, d\xi \\
		& \qquad \qquad \qquad  \qquad -\frac{1}{\kappa}\int_{\{\normxi \geqslant M\}}\frac{d\xi}{\normxi^{n+\Izero}}.
		\end{align*}
	    The third integral is clearly finite. On the set $\{\xi: \normxi \geqslant M \}$, we have $|\Psi(\xi)| \lesssim \normxi^{2}$. Hence, the first two integrals are bounded above by
	    \begin{equation*}
	    		 \frac{1}{\kappa(1-\kappa)}\int_{\{\normxi > M\}}\frac{d\xi}{\normxi^{n+\Izero-2\kappa}}
	    \end{equation*}
		Therefore $\int_{0}^{1}\frac{\mathcal{I}(t)}{t^{1+\kappa}} \, dt <\infty$ if $\Izero-2\kappa > 0 $ i.e., when $\kappa \in (0, \nicefrac{\Izero}{2})$. As $\mathcal{I}(t)$ is non-decreasing in $t$, we can use Lemma A.3 of \cite{khoshnevisan2022optimal} to conclude that $\mathbf{E}\left( \|X_{t}\|^{\Izero}\right) \lesssim t^{\kappa}$ for $\kappa \in (0, \nicefrac{\Izero}{2})$.
	\end{proof}

    \noindent Using the Lemmas proved above, we state and prove the main result of this section: the $L^1$ bounds on the transition densities 

	\begin{lemma}{\label{lem:L1_bounds}} Fix $T>0$. For all $\theta, \tau, \omega \in (0,1)$ and $t \in (0,T)$, we have
		\begin{equation}{\label{eqn:L1_bounds_space}}
		\int_{\mathbb{R}^{n}}|p_{t}(x+h)-p_{t}(x)| \, dx \lesssim \left(\frac{\|h\|^{\theta}}{t^{\frac{\theta+n}{\Iinf}+\kappa}}\right)^{\frac{\tau \Izero}{n+\Izero}} \, ,
		\end{equation}
		\begin{equation}{\label{eqn:L1_bounds_time}}
		\int_{\mathbb{R}^{n}}|p_{t+\varepsilon}(x)-p_{t}(x)| \, dx \lesssim \left(\frac{\varepsilon^{\theta}}{t^{\frac{2\theta+n}{\Iinf}+\kappa}}\right)^{\frac{\omega \Izero}{n+ \Izero}} \, ,
		\end{equation}
		uniformly for all $h \in \R^n$ such that $\|h\| \leqslant 1$ and $\varepsilon>0$.
	\end{lemma}
	\begin{proof}
		We first show equation \eqref{eqn:L1_bounds_space}. For a fixed $N \in \R$,  Write the integral as
		\begin{equation*}
			\int_{\{ \|x\| \leqslant N+1\}}|p_{t}(x+h) - p_{t}(x)| \, dx + \int_{\{ \|x\| > N+1\}}|p_{t}(x+h) - p_{t}(x)| \, dx
		\end{equation*}
	    For the first integral, we use equation \eqref{eqn:spatial_sup_bound} from Lemma \ref{lem:smoothtransitiondensitities} for the case $k=0$ to obtain
	    \begin{equation*}
	    	\int_{\{ \|x\| \leqslant N+1\}}|p_{t}(x+h) - p_{t}(x)| \, dx \lesssim N^n \sup_{x \in \R^n}|p_{t}(x+h) - p_{t}(x)| \lesssim \frac{N^n\|h\|^{\theta}}{t^\frac{\theta+n}{\Iinf}}
	    \end{equation*}
        For the other integral, we use Chebyshev's inequality:
		\begin{equation*}
		\int_{\{\|x\| > N+1\}}\left|p_t(x+h)-p_t(x)\right| \, dx \leqslant \mathbf{P}\left(\left\|X_t-h\right\|>N+1\right)+\mathbf{P}\left(\left\|X_t\right\|>N+1\right) \leqslant \frac{2\mathbf{E}(\|X_{t}\|^{\Izero})}{N^{\Izero}} \, .
		\end{equation*}
		From Lemma \ref{lem:moment_growth2}, we know $\mathbf{E}(\|X_{t}\|^{\Izero}) \lesssim t^{\kappa} \, \, \forall \, \kappa \in (0, \nicefrac{\Izero}{2})$. Therefore, we can combine the above two inequalities to yield: 
		\begin{equation*}
		\int_{\mathbb{R}^{n}}|p_{t}(x+h)-p_{t}(x)| \, dx \lesssim \frac{N^n\|h\|^{\theta}}{t^\frac{\theta+n}{\Iinf}}+\frac{t^{\kappa}}{N^{\Izero}},
		\end{equation*}
		uniformly for all $h \in \R^n$ such that $\|h\| \leqslant 1$. The optimal bound is attained when $\frac{N^n\|h\|^{\theta}}{t^\frac{\theta+n}{\Iinf}}=\frac{t^{\kappa}}{N^{\Izero}}$ which gives $N=\left(\frac{t^{\frac{\theta+n}{\Iinf}+\kappa}}{\|h\|^{\theta}}\right)^{\frac{1}{n+\Izero}}$. 
		Of course there is also the trivial bound:
	    \begin{equation}
	    	\int_{\R^n}|p_{t}(x+h)-p_{t}(x)| \, dx \leqslant 2 \, .
	    \end{equation} 
	    Hence we have:
		\begin{equation*}
		\int_{\mathbb{R}^{n}}|p_{t}(x+h)-p_{t}(x)| \, dx \lesssim \left(\frac{\|h\|^{\theta}}{t^{\frac{\theta+n}{\Iinf}+\kappa}}\right)^{\frac{\Izero}{n+\Izero}} \wedge 1  \lesssim \left(\frac{\|h\|^{\theta}}{t^{\frac{\theta+n}{\Iinf}+\kappa}}\right)^{\frac{\tau\Izero}{n+\Izero}} \, ,
		\end{equation*}
		for all $\tau \in (0,1)$. This proves equation \eqref{eqn:L1_bounds_space}. Equation \eqref{eqn:L1_bounds_time} follows similarly, but now we use equation \eqref{eqn:temporal_sup_bound} instead, along with Lemma \ref{lem:moment_growth2}. 
	\end{proof}

	\section{\textbf{H\"{o}lder Continuity}}{\label{sec:holdercontinuity}}
	In this section, we use the results from the previous sections to prove Theorem \ref{thm:main_thm}. 
	Following the approach of Dalang \cite{dalang1999extending} and Walsh \cite{walsh1986introduction}, 
	we begin by defining a random-field solution to \eqref{eqn:spde}.  To wit: let $\{F_{t}(A): t\in [0, \infty) \text{ and } \lambda(A) < \infty \}$ be the martingale measure associated to the noise $\dot{F}$, and $\mathcal{F}_{t} \coloneqq \sigma( F_{s}(A): 0 \leqslant s \leqslant t \text{ and } \lambda(A) < \infty$). An $\mathcal{F}_{t}$-measurable process $\{u(t,x): t \in [0,\infty), x \in \mathbb{R}^{n}\}$ is a solution to \eqref{eqn:spde} if
	\begin{equation}{\label{eqn:spdesolution}}
	u(t,x)=(p_{t} \ast u_{0})(x)+\int_{0}^{t}\int_{\R^n}p_{t-s}(y-x)b(u(s,y)) \, ds \,dy+\int_{0}^{t}\int_{\R^n}p_{t-s}(y-x)\sigma(u(s,y)) \, F(ds \,dy)
	\end{equation}
    Moreover, such a solution satisfies the uniform bound
	\begin{equation}{\label{eqn:momentbound}}
	\sup_{t\in [0,T]} \sup_{x\in \mathbb{R}^{n}}\mathbf{E}(\lvert u(t,x) \rvert^{p}) < \infty \quad \forall p \geqslant1 \text{ and } T>0
	\end{equation}
	We further assume that the exponent of H\"older continuity of $u_{0}$, $\rho$ satisfies $0<\rho < \Izero$. Otherwise, the H\"older indices of the solutioin will just be in terms of $\rho \wedge \Izero$. Lemma \ref{lem:momentequivalence} then guarantees that $\mathbf{E}\left( \|X_{t}\|^{\rho}\right) < \infty$. 
	
	We are now ready to prove Theorem \ref{thm:main_thm}. The first part follows from the following Lemma on the equivalence of aforementioned sufficient conditions for continuity.
	\begin{lemma}{\label{lem:equivalent_conditions}} Under Assumption \ref{assumption1}, the following are equivalent:
		\begin{compactenum}
			\item $\int_{\mathbb{R}^{n}} \frac{|\Psi(\xi)|^{\gamma}\mu(d\xi)}{1+\repsi(\xi)} < \infty \quad \text{for some } \gamma \in (0,1)$ ;\\ \label{condition1}
			\item $\int_{\mathbb{R}^{n}}\frac{\| \xi \|^{2\delta} \mu(d\xi)}{1+\repsi(\xi)} < \infty \quad \text{for some } \delta \in (0,1)$ ;\\ \label{condition2}
			\item $\int_{\mathbb{R}^{n}}\frac{\mu(d\xi)}{(1+\repsi(\xi))^{1-\eta}}<\infty \quad \text{for some } \eta \in (0,1)$ . \label{condition3}
		\end{compactenum}
    Therefore, $\indl = \indm = \indu$.
	\end{lemma}
	\begin{proof}
    	We first show that \eqref{condition1} $\implies$ \eqref{condition3}. Fix $\eta \in (0,\gamma)$. By Assumption \ref{assumption1}, $\{\xi : |\Psi(\xi) \leqslant 1\}$ is bounded. On this set, we have $(1+\repsi(\xi))^\eta \leqslant (1+|\Psi(\xi)|)^{\eta} \leqslant 2^{\eta}$. Therefore,
		\begin{equation*}
		\int_{\{\xi: |\Psi(\xi)| \leqslant 1\}}\frac{\mu(d\xi)}{(1+\repsi(\xi))^{1-\eta}}=\int_{\{\xi: |\Psi(\xi)| \leqslant 1\}}\frac{(1+\repsi(\xi))^{\eta}}{1+\repsi(\xi)}\, \mu(d\xi) \leqslant 2^{\eta}\mu(\{\xi: |\Psi(\xi)| \leqslant 1\})  < \infty \, .
		\end{equation*} 
		since $\mu$ is tempered and hence $\sigma-$finite. On the complementary set $\{\xi:|\Psi(\xi)|>1\}$, we note that $(1+\repsi(\xi))^{\eta} \leqslant 2^{\eta}|\Psi(\xi)|^{\eta} \leqslant 2^{\eta}|\Psi(\xi)|^{\gamma}$. Thus,
		\begin{align*}
		\int_{\{\xi: |\Psi(\xi)| > 1\}}\frac{\mu(d\xi)}{(1+\repsi(\xi))^{1-\eta}}
		&=\int_{\{\xi: |\Psi(\xi)| > 1\}}\frac{(1+\repsi(\xi))^{\eta}}{1+\repsi(\xi)} \,\mu(d\xi) \\
		&\leqslant 2^{\eta}\int_{\{\xi: |\Psi(\xi)| > 1\}}\frac{|\Psi(\xi)|^{\gamma}}{1+\repsi(\xi)}  \, \mu(d\xi) < \infty \,  .
		\end{align*}
	    This proves $\eqref{condition1} \implies \eqref{condition3}$. 
	    
	    Next we show the implication $\eqref{condition3} \implies \eqref{condition2}$. Outside a neighborhood of zero, say on $\{\|\xi\| > M\}$, we have again by Assumption \ref{assumption1},  $\repsi(\xi) \geqslant\|\xi\|^{\Iinf}$ and so the inequality $(1+\repsi(\xi))^{\eta} \geqslant\|\xi\|^{\eta \Iinf}$ holds. If we choose $\delta \in (0,\frac{\eta\Iinf}{2})$, then we obtain
		\begin{equation*}
		\int_{\{\|\xi\|>M\}}\frac{\|\xi\|^{2\delta}}{1+\repsi(\xi)} \,  \, \mu(d\xi) \leqslant \int_{\{\|\xi\|>M\}}\frac{(1+\repsi(\xi))^{\eta}}{1+\repsi(\xi)} \, \mu(d\xi) =\int_{\{\|\xi\|>M\}}\frac{\mu(d\xi)}{(1+\repsi(\xi))^{1-\eta}} < \infty .
		\end{equation*}
		On the bounded set $\{\|\xi\| \leqslant M\}$, the integrand is continuous, so
		\begin{equation*}
		\int_{\{\|\xi\| \leqslant M\}}\frac{\|\xi\|^{2\delta}}{1+\repsi(\xi)} \, \mu(d\xi) \leqslant M^{2\delta}\mu(\{\|\xi\| \leqslant M\}) < \infty .
		\end{equation*}
		Therefore, \eqref{condition3} implies \eqref{condition2}. 
		
		Finally, we establish $\eqref{condition2} \implies \eqref{condition1}$. By the growth bound from equation \eqref{eqn:khinchtineestimate}, we have $|\Psi(\xi)| \lesssim (1+\|\xi\|^{2}) \, \forall \xi \in \mathbb{R}^{n}$. Thus, on the set $\{\|\xi\| \leqslant 1\}$, the inequality $|\Psi(\xi)|^{\gamma} \lesssim (1+\|\xi\|^{2})^{\gamma} \lesssim 2^{\gamma}$ holds. So,
		\begin{equation*}
		\int_{\{\|\xi\| \leqslant 1\}} \frac{|\Psi(\xi)|^{\gamma}}{1+\repsi(\xi)} \, \mu(d\xi) \lesssim 2^{\gamma}\mu(\{\|\xi\| \leqslant 1\}) < \infty \,.
		\end{equation*}
		On the complement set $\{\|\xi\| > 1\}$, $|\Psi(\xi)|^{\gamma} \lesssim (1+\|\xi\|^{2})^{\gamma} \lesssim 2^{\gamma} \|\xi\|^{2\gamma}$. Hence,
		\begin{equation*}
		\int_{\{\|\xi\| > 1\}}\frac{|\Psi(\xi)|^{\gamma}}{1+\repsi(\xi)} \, \mu(d\xi) \lesssim 2^{\gamma}\int_{\{\|\xi\|>1\}} \frac{\|\xi\|^{2\gamma}}{1+\repsi(\xi)} \, \mu(d\xi) < \infty \, ,
		\end{equation*}
		for all $\gamma \in (0, \delta)$. This proves the implication $\eqref{condition2} \implies \eqref{condition1}$, thus concluding the proof.
	\end{proof}
	\noindent Now we are ready to prove the second half of Theorem \ref{thm:main_thm}.
	\subsection{Regularity In Space}
    In this subsection and the next, we derive the $L^p-$estimates of the increments of the solution needed to invoke the Kolmogorov continuity theorem. To this end, we start with the spatial increments. From equation \ref{eqn:spdesolution} and triangle inequality, we have $\forall p>1$:
	\begin{align}
	\mathbf{E}\Bigg(\lvert u(t,x+h)
    &-u(t,x) \rvert^{2p}\Bigg) \lesssim \mathbf{E}\left( \lvert (p_{t} \ast u_{0})(x+h)-(p_{t} \ast u_{0})(x) \rvert^{2p} \right) \nonumber \\ 
   	&+\mathbf{E}\left(\left|\int_{0}^{t}\int_{\mathbb{R}^{n}}\left[p_{t-s}(x+h-y)-p_{t-s}(x-y)\right]b(u(s,y)) \, ds \, dy \right|^{2p}\right)  \nonumber \\
	&+\mathbf{E}\left(\left|\int_{0}^{t}\int_{\mathbb{R}^{n}}\left[p_{t-s}(x+h-y)-p_{t-s}(x-y)\right]\sigma(u(s,y)) \, F(ds \, dy)\right|^{2p}\right)  \label{eqn:spatial_increments}
	\end{align}
    For the second term on the right-hand side, we first rewrite the integrand as
    \begin{equation*}
    	|p_{t-s}(y-x-h)- p_{t-s}(y-x)|^{\frac{2p-1}{2p}}|p_{t-s}(y-x-h)- p_{t-s}(y-x)|^{\frac{1}{2p}}|b(u(s,y))|
    \end{equation*}
    This factorization allows us to apply H\"older’s inequality with exponents $\tfrac{2p}{2p-1}$ and $2p$. Using this, together with the Lipschitz property of $b$ whence
    \begin{equation*}
    	\E(|b(u(s,y))|^{2p}) \lesssim 1 + \E(|u(s,y)|^{2p}) < \infty,
    \end{equation*}
    where finiteness follows directly from equation \eqref{eqn:momentbound}. Doing so, we get
    \begin{equation*}
    	\left(\int_{0}^{t}\int_{\R^n}|p_{t-s}(y-x-h) - p_{t-s}(y-x)| \, ds \, dy \right)^{2p}
    \end{equation*}
    After a change of variables $s \to t-s$, we use Lemma \ref{lem:L1_bounds} to see that
    \begin{equation}{\label{eqn:driftspatialincrement}}
    	\mathbf{E}\left(\left|\int_{0}^{t}\int_{\mathbb{R}^{n}}\left[p_{t-s}(x+h-y)-p_{t-s}(x-y)\right]b(u(s,y)) \, ds \, dy \right|^{2p}\right) \lesssim \|h\|^{2\alpha p} \left(\int_{0}^{t}\frac{ds}{s^{\eta}} \right)^{2p},
    \end{equation}
    where we set
    \begin{equation*}
    	\alpha \coloneqq \frac{\theta \tau \Izero}{n+\Izero} \quad \text{and} \quad\eta  = \left(\frac{n + \theta}{\Iinf} + \kappa\right) \cdot \frac{\tau \Izero}{n + \Izero}, \qquad \theta, \tau \in (0,1), \quad \kappa \in (0, \nicefrac{\Izero}{2})
    \end{equation*}
    Because $\theta, \tau $and $\kappa$ are free parameters, the range of theses quantities are
    \begin{equation*}
    	\left(\frac{n + \theta}{\Iinf} + \kappa\right) \in \left(\frac{n}{\Iinf}, \frac{n + 1}{\Iinf} + \frac{\Izero}{2}\right) \quad \text{and} \quad \frac{\tau \Izero}{n + \Izero} \in \left(0, \frac{\Izero}{n + \Izero}\right)
    \end{equation*}
    Hence, $\eta \in (0, K)$ where
    \begin{equation*}
    	K = \left(  \frac{n + 1}{\Iinf} + \frac{\Izero}{2}\right) \cdot \frac{\Izero}{n + \Izero}
    \end{equation*}
    To keep the right-hand side of equation \eqref{eqn:driftspatialincrement} to be finite,we also need $\eta \in (0,1)$. Thus we choose $\eta \in (0, \min\{K, 1\})$.
     
    For the parameter $\alpha$, we first note that
    \begin{equation*}
    	 \alpha < \frac{\Izero}{n+\Izero} < 1.
    \end{equation*}
    Moreover, $\alpha$ can be rewritten as
    \begin{equation*}
    	\alpha = \eta \cdot \frac{\theta}{\frac{\theta + n}{\Iinf} + \kappa} < \eta \cdot \frac{\theta}{\frac{\theta + n}{\Iinf}} = \eta \cdot \frac{\theta \Iinf}{\theta + n} < \frac{\eta \Iinf}{n + 1}
    \end{equation*}
    Putting these together, we obtain $0 < \alpha < \min\left\{ \frac{\Izero}{n + \Izero}, \frac{\eta \Iinf}{n + 1} \right\} < \min\left\{ \frac{\Izero}{n + \Izero}, \frac{ \min\{K, 1\}\Iinf}{n + 1} \right\} < 1$.
    
	For the third term on the right-hand side of equation \eqref{eqn:spatial_increments}, we begin by applying the Burkholder-Davis-Gundy inequality \cite{revuz2013continuous} followed by a change of variables $s \to t-s$ to see that
	\begin{align}{\label{eqn:BDGinequality}}
	\mathbf{E} \Bigg(\bigg| \int_{0}^{t}\int_{\mathbb{R}^{n}}  \big[p_{t-s}(y-x-h)-& p_{t-s}(y-x)\big]  \sigma(u(s,y)) \, F(ds \, dy) \bigg|^{2p}\Bigg) \nonumber \\
    & \lesssim \Bigg(\int_{0}^{t}\int_{\mathbb{R}^{n}}\int_{\mathbb{R}^{n}}\mathcal{G}_{s,s}(z,h)  \mathcal{G}_{s,s}(z-y,h) \,ds \,\Gamma(dy) \,dz\Bigg)^{p}
  	\end{align}
  
	where, for ease of notation we set
	\begin{equation*}
		\mathcal{G}_{r,s}(z,h)\coloneqq \lvert p_{r}(z-x-h)-p_{s}(z-x)\rvert \lvert \sigma(u(s,y)) \rvert.
	\end{equation*} 
    Applying H\"older's inequality first with respect to $dy$ and then with $\mathbf{P}$, followed by equation \eqref{eqn:momentbound}, the right hand side can be bounded as follows:
	\begin{align*}
	\Bigg(
	&\int_{0}^{t}
	 \int_{\mathbb{R}^{n}}\int_{\mathbb{R}^{n}} \mathcal{G}_{s}(z) \mathcal{G}_{s}(z-y) \,ds \,\Gamma(dy) \,dz\Bigg)^{p} \\
	& \Bigg(\lesssim \int_{0}^{t} \int_{\R^n} \int_{\R^n} \lvert p_{s}(z-x-h)-p_{s}(z-x)\rvert \lvert p_{s}(z-y-x-h)-p_{s}(z-y-x)\rvert \, \Gamma(dy) \, dz \, ds \Bigg)^{p}
	\end{align*}
    As the densities are non-negative, $\lvert p_{s}(z-y-x-h)-p_{s}(z-y-x)  \rvert \leqslant  p_{s}(z-y-x-h)+p_{s}(z-y-x)$. We substitute this into the equation above, integrate out the $y$ variable, and use relation \eqref{eqn:bochnerminlosschwarz}:
    \begin{equation*}
    \left| \int_{\R^n}  p_{s}(z-y-x-h)+p_{s}(z-y-x) \, \Gamma(dy)\right| \lesssim \int_{\R^n} e^{-s \repsi} \, \mu(d\xi)
    \end{equation*}
    The above equality warrants an explanation as $p_{s}$ does not necessarily belong to $\mathcal{S}(\R^n)$ and so a direct application of the Bochner-Minlos-Schwartz theorem is  not allowed. By an approximation argument, the identity \eqref{eqn:bochnerminlosschwarz} can be extended to all $\varphi$ such that $\widehat{\varphi} \in \mathcal{L}^{2}(\R^n,\mu)$. As for the $dz$ integral, using Lemma \ref{lem:L1_bounds},
	\begin{equation*}
	\int_{\mathbb{R}^{n}}\lvert p_{s}(z-x-h)-p_{s}(z-x) \rvert \, dz \lesssim \left(\frac{\|h\|^{\theta}}{s^{\frac{\theta+n}{\Iinf}+\kappa}}\right)^{\frac{\tau \Izero}{n+\Izero}} \, .
	\end{equation*}
	Putting everything together, and writing the result in terms of the parameters $\alpha$ and $\eta$ defined earlier, we obtain
	\begin{align*}
	\mathbf{E}\Bigg(\bigg| \int_{0}^{t}\int_{\mathbb{R}^{n}}\big[p_{s}(y-x-h)-p_{s}(y-x)\big]\sigma(u(s,y)) \, 
	& F(ds \, dy) \bigg|^{2p}\Bigg) \\ 
	& \lesssim \|h\|^{p\alpha}\left(\int_{0}^{t}\int_{\mathbb{R}^{n}}\frac{e^{-s\repsi(\xi)}}{s^{\eta}} \, \mu(d\xi) \, \, ds \right)^{p}
	\end{align*} 
    To evaluate the inner integral, we use Fubini-Tonelli theorem:
	\begin{align*}
	\int_{0}^{t}\int_{\mathbb{R}^{n}}\frac{e^{-s\repsi(\xi)}}{s^{\eta}} \, \mu(d\xi) \, ds&=
	\int_{\mathbb{R}^{n}} \, \mu(d\xi)\left(\int_{0}^{t}\frac{e^{s}e^{-s(1+\repsi(\xi))}}{s^{\eta}} \,ds\right) \\
	& \leqslant e^{T}\int_{\mathbb{R}^{n}}\frac{\mu(d\xi)}{(1+\repsi(\xi))^{1-\eta}}\left( \int_{0}^{t(1+\repsi(\xi))}\frac{e^{-r}}{r^\eta} \, dr\right) \\
	& \lesssim e^{T}\int_{\mathbb{R}^{n}}\frac{\mu(d\xi)}{(1+\repsi(\xi))^{1-\eta}}\left( \int_{0}^{\infty}\frac{e^{-r}}{r^\eta} \, dr\right)\\
	&= \int_{\mathbb{R}^{n}}\frac{\mu(d\xi)}{(1+\repsi(\xi))^{1-\eta}} \, .
	\end{align*}
    The last integral is finite for all $\eta \in (0, \indu)$. Since by definition, $\eta \in (0, \min\{K,1\})$, we conclude $\eta \in (0, \min\{\indu, K\})$, as $\indu \leqslant 1$.
    
	Now let us look at the term involving the initial condition in equation \eqref{eqn:spatial_increments}. We have
	\begin{align*}
	\left \lvert (p_{t} \ast u_{0})(x+h)-(p_{t} \ast u_{0})(x) \right \rvert &= \left\lvert \int_{\mathbb{R}^{n}}p_{t}(y)u_{0}(x+h-y) \, dy-\int_{\mathbb{R}^{n}}p_{t}(y)u_{0}(x-y) \, dy \right \rvert \\
	& \lesssim \int_{\mathbb{R}^{n}}p_{t}(y)\left \lvert u_{0}(x+h-y)-u_{0}(x-y) \right \rvert \,  dy  \\
	& \lesssim \int_{\mathbb{R}^{n}}p_{t}(y)\|h\|^{\rho} \, dy \\
	& \lesssim \|h\|^{\rho} \, ,
	\end{align*}
	so that
	\begin{equation*}
	\mathbf{E} \left(  \left \lvert (p_{t} \ast u_{0})(x+h)-(p_{t} \ast u_{0})(x) \right \rvert^{2p} \right) \lesssim  \|h\|^{2\rho p} \, .
	\end{equation*}
	Putting this together with the earlier bounds, we conclude that
	\begin{equation}{\label{eqn:spatial_regularity}}
	\sup_{[0,T]}\sup_{x \in \mathbb{R}^{n}}\mathbf{E}\left(\lvert u(t,x+h)-u(t,x) \rvert^{2p}\right) \lesssim \|h\|^{(\alpha \wedge 2\rho) p} \, ,
	\end{equation}
	for all 
	\begin{equation}{\label{eqn:alpha_range}}
		0 < \alpha < \min\left\{ \frac{\Izero}{n + \Izero}, \frac{\eta \Iinf}{n + 1} \right\} < \min\left\{ \frac{\Izero}{n + \Izero}, \frac{ \min\{\indu, K\}\Iinf}{n + 1} \right\} 
	\end{equation} 
    where $K = \left(  \frac{n + 1}{\Iinf} + \frac{\Izero}{2}\right) \cdot \frac{\Izero}{n + \Izero}$.
    
    \subsection{Regularity In Time}
	We now study the temporal increments of the solution. Starting from equation \eqref{eqn:spdesolution}, we have
	\begin{align}
		\mathbf{E}\Bigg(\lvert u(t+\varepsilon,x)
		&-u(t,x) \rvert^{2p}\Bigg) \lesssim \mathbf{E}\left( \lvert (p_{t+\varepsilon} \ast u_{0})(x)-(p_{t} \ast u_{0})(x) \rvert^{2p} \right) \nonumber \\ 
		&+\mathbf{E}\left(\left|\int_{0}^{t}\int_{\mathbb{R}^{n}}\left[p_{t+\varepsilon-s}(y-x)-p_{t-s}(y-x)\right]b(u(s,y)) \, ds \, dy \right|^{2p}\right)  \nonumber \\
		&+\mathbf{E}\left(\left|\int_{0}^{t}\int_{\mathbb{R}^{n}}\left[p_{t+\varepsilon-s}(y-x)-p_{t-s}(y-x)\right]\sigma(u(s,y)) \, F(ds \, dy)\right|^{2p}\right) \nonumber \\
		& + \mathbf{E}\left(\left| \int_{t}^{t+\varepsilon}\int_{\mathbb{R}^{n}}p_{t+\varepsilon-s}(y-x)b(u(s,y)) \, ds \, dy\right|^{2p}\right) \nonumber \\ 
		& +\mathbf{E}\left(\left| \int_{t}^{t+\varepsilon}\int_{\mathbb{R}^{n}}p_{t+\varepsilon-s}(y-x)\sigma(u(s,y)) \, F(ds \, dy)\right|^{2p}\right) \label{eqn:temporal_increments}
	\end{align}
	For the first expectation, note that
	\begin{align*}
	\big\lvert (p_{t+\varepsilon} & \ast u_{0})(x)-(p_{t}  \ast u_{0})(x) \big\rvert  = \left\lvert \int_{\mathbb{R}^{n}}p_{t+\varepsilon}(y)u_{0}(x-y) \, dy-\int_{\mathbb{R}^{n}}p_{t}(y)u_{0}(x-y) \, dy \right \rvert \\
	&= \left \lvert \int_{\mathbb{R}^{n}}\int_{\mathbb{R}^{n}}p_{\varepsilon}(z)p_{t}(y-z)u_{0}(x-y) \, dy \,dz- \int_{\mathbb{R}^{n}}p_{\varepsilon}(z)dz\int_{\mathbb{R}^n}p_{t}(y)u_{0}(x-y)dy \right \rvert \\
	&= \left \lvert \int_{\mathbb{R}^{n}}\int_{\mathbb{R}^{n}}p_{\varepsilon}(z)p_{t}(y-z)u_{0}(x-y)dydz- \int_{\mathbb{R}^{n}}\int_{\mathbb{R}^n}p_{\varepsilon}(z)p_{t}(y-z)u_{0}(x-y+z)dydz \right \rvert \\
	& \lesssim \int_{\mathbb{R}^{n}}p_{\varepsilon}(z) \left(\int_{\mathbb{R}^{n}}p_{t}(y-z)\lvert u_{0}(x-y)-u_{0}(x-y+z)\rvert dy\right)dz  \\
	&\lesssim\int_{\mathbb{R}^{n}}p_{\varepsilon}(z)\|z\|^{\rho}dz= \mathbf{E}(\|X_{\varepsilon}\|^{\rho}) \lesssim \varepsilon^{\frac{\rho}{2}} \, ,
	\end{align*}
	by Lemma \ref{lem:moment_growth2}. It follows that
	\begin{equation*}
	\mathbf{E}\left( \left \lvert (p_{t+\varepsilon} \ast u_{0})(x)-(p_{t} \ast u_{0})(x) \right \rvert^{2p} \right) \lesssim \varepsilon^{\rho p} \, .
	\end{equation*}
    The second expectation in equation \eqref{eqn:temporal_increments} is handled as in the spatial case, using Lemma \ref{lem:L1_bounds}. This gives
     \begin{equation}{\label{eqn:drifttemporalincrement}}
    	\mathbf{E}\left(\left|\int_{0}^{t}\int_{\mathbb{R}^{n}}\left[p_{t+\varepsilon-s}(y-x)-p_{t-s}(y-x)\right]b(u(s,y)) \, ds \, dy \right|^{2p}\right) \lesssim |\varepsilon|^{2\alpha p} \left(\int_{0}^{t}\frac{ds}{s^{\widetilde{\eta}}} \right)^{2p},
    \end{equation}
    where now
    \begin{equation*}
    	\beta \coloneqq \frac{\theta \omega \Izero}{n+\Izero} \quad \text{and} \quad \widetilde{\eta} \coloneqq \left(\frac{n + 2\theta}{\Iinf} + \kappa\right) \cdot \frac{\omega \Izero}{n + \Izero}, \qquad \theta, \omega \in (0,1), \quad \kappa \in (0, \nicefrac{\Izero}{2})
    \end{equation*}
    and by the same reasoning as in the previous section, we obtain $\widetilde{\eta} \in (0, \min\{\widetilde{K},1\})$ with
    \begin{equation*}
    	\widetilde{K} = \left(  \frac{n + 2}{\Iinf} + \frac{\Izero}{2}\right) \cdot \frac{\Izero}{n + \Izero},
    \end{equation*}
	 and $0 < \beta < \min\left\{ \frac{\Izero}{n + \Izero}, \frac{\widetilde{\eta} \Iinf}{n + 2} \right\} < \min\left\{ \frac{\Izero}{n + \Izero}, \frac{ \min\{\widetilde{K}, 1\}\Iinf}{n + 2} \right\} < 1$. 
	 The third and fifth expectations involve stochastic integrals. Using the Burkholder-Davis-Gundy and H\"older's inequalities and again, with Lemma \ref{lem:L1_bounds}, we get 
	\begin{align*}
	\mathbf{E}\Bigg(\Bigg| \int_{0}^{t}\int_{\mathbb{R}^{n}}\left[p_{t+\varepsilon-s}(y-x)-p_{t-s}(y-x)\right]\sigma(u(s,y)) \, F(ds \, dy) 
	& \Bigg|^{2p}\Bigg) \lesssim \Bigg(\int_{0}^{t}\int_{\mathbb{R}^{n}}\int_{\mathbb{R}^{n}} \mathcal{G}_{s+\varepsilon,s}(z,0) \\
	& \cdot \mathcal{G}_{s+\varepsilon,s}(z-y,0) \,ds \,\Gamma(dy) \,dz\Bigg)^{p} \\
	& \lesssim |\varepsilon|^{{p\beta}} \Bigg(\int_{\mathbb{R}^{n}}\frac{\mu(d\xi)}{(1+\repsi(\xi))^{1-\widetilde{\eta}}}\Bigg)^{p}
	\end{align*} 
    and
	\begin{align*}
	\mathbf{E}\left(\left| \int_{t}^{t+\varepsilon}\int_{\mathbb{R}^{n}}p_{t+\varepsilon-s}(y-x)\sigma(u(s,y)) \, F(ds \, dy)\right|^{2p}\right) 
	&\lesssim  \Bigg(\int_{0}^{\varepsilon}\int_{\mathbb{R}^{n}}\int_{\mathbb{R}^{n}}p_{s}(z-x) \\
	& \qquad \qquad \cdot p_{s}(z-y-x)\,ds \, \Gamma(dy) \, dz \Bigg)^{p} \\
	& \lesssim \Bigg(\int_{0}^{\varepsilon}\int_{\mathbb{R}^{n}}e^{-2s\repsi(\xi)} \, \mu(d\xi)\Bigg)^{p} \\
	& \lesssim |\varepsilon|^{p\beta}\Bigg(\int_{\mathbb{R}^{n}}\frac{|\Psi(\xi)|^{\beta}\mu(d\xi)}{1+\repsi(\xi)} \Bigg)^{p}
	\end{align*}
    where $\beta$ and $\widetilde{\eta}$ are as defined before. 
    
    For the fourth expectation, we directly estimate
    \begin{align*}
    	\mathbf{E}\left(\left| \int_{t}^{t+\varepsilon}\int_{\mathbb{R}^{n}}p_{t+\varepsilon-s}(y-x)b(u(s,y)) \, ds \, dy)\right|^{2p}\right) 
    	& \lesssim \left(\int_{0}^{\varepsilon}\int_{\R^n}p_{t+\varepsilon-s}(y-x) \, ds \, dy \right)^{2p} \\
        & \lesssim |\varepsilon|^{2p}
    \end{align*}
	Combining the above bounds, we arrive at
	\begin{equation}{\label{eqn:temporal_regularity}}
		\sup_{[0,T]}\sup_{x \in \mathbb{R}^{n}}\mathbf{E}\left(\lvert u(t+\varepsilon,x)-u(t,x) \rvert^{2p}\right) \lesssim |\varepsilon|^{(\beta \wedge \rho) p} \, ,
	\end{equation}
	for all 
	\begin{equation}{\label{eqn:beta_range}}
		0 < \beta < \min\left\{ \frac{\Izero}{n + \Izero}, \frac{\widetilde{\eta} \Iinf}{n + 2} \right\} < \min\left\{ \frac{\Izero}{n + \Izero}, \frac{ \min\{\indu, \widetilde{K}\}\Iinf}{n + 1} \right\} 
	\end{equation} 
    Finally, putting together equations \eqref{eqn:spatial_regularity} and \eqref{eqn:temporal_regularity}, we can apply Kolmogorov's Continuity Theorem to obtain a version of the solution $u$ in $\mathcal{C}^{\nicefrac{\alpha}{2}, \nicefrac{\beta}{2}}_{\operatorname{loc}}$ for all $\alpha$ and $\beta$ in the ranges described in equations \eqref{eqn:alpha_range} and \eqref{eqn:beta_range} respectively. This completes the proof of Theorem \ref{thm:main_thm}.
	\section{Acknowledgements}
	The author would like to thank his advisor Davar Khoshnevisan for suggesting the problem and his invaluable input and support. The author would also like to thank Tom Alberts, Li-Cheng Tsai for a careful reading of an earlier draft and providing suggestions for improvements.
	

\end{document}